\documentclass[11 pt,a4paper]{article}

\usepackage[utf8]{inputenc}
\usepackage[english]{babel}

\usepackage{amsmath}
\usepackage{mathtools}
\usepackage{amsfonts}
\usepackage{amssymb}
\usepackage{amsthm}

\usepackage{stmaryrd}
\usepackage{mathabx}

\usepackage[usenames,dvipsnames,pdftex]{xcolor}

\usepackage{enumitem}

\usepackage{cleveref}
\usepackage{todonotes}
\usepackage{empheq}

\newtheorem{theorem}{Theorem}

\newtheorem{corollary}[theorem]{Corollary}
\newtheorem{proposition}[theorem]{Proposition}

\theoremstyle{definition}

\theoremstyle{remark}

\newcommand{\comment}[1]{{}}

\newcommand\ZZ{\mathbb{Z}}

\newcommand{\familysty}[1]{\mathcal{#1}}
\newcommand{\familyF}{\familysty{F}}

\newcommand{\tr}{^{\top}}
\newcommand{\GL}[2]{\mathsf{GL}_{#1}(#2)}

\newcommand{\groupi}{G}          
\newcommand{\fatf}{\ZZ^m \times F_n}
\newcommand{\fatfab}{G}

\newcommand{\ab}{\operatorname{^{ab}}}
\newcommand{\End}[1]{\mathsf{End}\,#1}

\newcommand{\Aut}[1]{\mathsf{Aut}\,#1}
\newcommand{\Mon}[1]{\mathsf{Mon}\,#1}
\newcommand{\Epi}[1]{\mathsf{Epi}\,#1}

\newcommand{\WhP}[2]{\operatorname{WhP}(#1,#2)}
\newcommand{\ti}{\mbox{type\,(I)}}
\newcommand{\tii}{\mbox{type\,(II)}}

\newcommand{\detm}[1]{\det #1}

\begin{document}
\title{Whitehead problems for words in $\ZZ^m \times F_n$}%
\author{
\textbf{Jordi Delgado}\thanks{
The author thanks the hospitality of the Centre de Recerca Matemàtica (CRM-Barcelona) along the research programme on Automorphisms of Free Groups during which this preprint was finished; and gratefully acknowledge the support of \emph{Universitat Polit\`{e}cnica de Catalunya} through the PhD grant number 81--727 and the MEC (Spain) through project number MTM2011-25955.
} \\[3pt]
{\em Dept. Mat. Apl. III}\\ {\em Universitat Polit\`ecnica de Catalunya,} \\ {\em Manresa, Barcelona} \\
\texttt{jorge.delgado@upc.edu}  \\
\texttt{jdelgado.upc@gmail.com}\\[15pt]
}
\maketitle
\begin{abstract}
We solve the Whitehead problem for automorphisms, monomorphisms and endomorphisms in~$\ZZ^m \times F_n$
after giving an explicit description of each of these families of transformations.
\end{abstract}

\maketitle

We generically call \emph{Whitehead problems} for a finitely presented group $\groupi$ the problems consisting in, given two objects (of the same certain suitable kind $\mathcal{O}$) in $\groupi$ and a family $\familyF$ of transformations, decide whether there exists an element in $\familyF$ sending one object to the other.
Specifically we will write
$
\WhP{\mathcal{O}}{\familyF}
$
to mean the Whitehead problem with objects in $\mathcal{O}$ and transformations in $\familyF$, i.e.
\[
\WhP{\mathcal{O}}{\familyF} \ \equiv \  \Large{\text{\textquestiondown}}
\exists \varphi \in \familyF \normalsize{\text{ such that }}  o_1 \overset{\varphi}{\mapsto} o_2
\Large{\text{?}} \,
_{(o_1,o_2 \text{ in } \mathcal{O})}.
\]
It is customary to include as a part of the problem, the search of one of such transformations, in case that there exists. So will we.

The ``\,objects in $\groupi$\,'' usually considered include elements (i.e.\ words in the generators), subgroups and conjugacy classes, as well as tuples of them; while the typical families of transformations are those of automorphisms, monomorphisms, epimorphisms and endomorphisms of $G$; we denote them respectively by $\Aut{\groupi}$, $\Mon{\groupi}$, $\Epi{\groupi}$ and $\End{\groupi}$.

Using this scheme, the first problem of this kind (proposed and solved by Whitehead in \cite{whitehead_equivalent_1936}) is $\WhP{F_n}{\Aut{F_n}}$, where $F_n$ denotes the free group on $n$ generators.

\bigskip

In this note we will deal with Whitehead problems for words in finitely generated free-abelian times free groups\comment{($\ZZ^m \times F_n$)} (see \cite{delgado_algorithmic_2013} for full details). In sake of notational easiness we will hereafter usually abbreviate $\fatfab = \fatf$.
Concretely we will solve $\WhP{\fatfab}{\Aut{\fatfab}}$, $\WhP{\fatfab}{\Mon{\fatfab}}$ and $\WhP{\fatfab}{\End{\fatfab}}$. It is not surprising that the (already solved) corresponding problems for $\ZZ^m$ and $F_n$ emerge when considering Whitehead problems for~$\fatfab\comment{ = \fatf}$.

For the free-abelian groups the problems considered become those of the existence of solutions (of certain type) for integer matrix equations of the form $\mathbf{a}\cdot \mathbf{X} = \mathbf{b}$. This can be easily decided using linear algebra.

\begin{proposition} \label{thm:Whiteheads Z^m}
Let $m \geq 1$, then
\begin{enumerate}
\item[\rm{(i)}] $\WhP{Z^m}{\Aut{Z^m}}$ is solvable.
\item[\rm{(ii)}] $\WhP{Z^m}{\Mon{Z^m}}$ is solvable.
\item[\rm{(iii)}] $\WhP{Z^m}{\End{Z^m}}$ is solvable. \qed
\end{enumerate}
\end{proposition}

The same problems for the free group $F_n$ are much more complicated. As
mentioned above, the case of automorphisms was already solved by Whitehead back in the 30's of the last century. The case of endomorphisms can be solved by writing a system of equations over $F_n$ (with unknowns being the images of a given free basis for $F_n$), and then solving it by the powerful Makanin's algorithm. Finally, the case of monomorphisms was recently solved by Ciobanu and Houcine.

\goodbreak

\begin{theorem}\label{thm:Whitehead Fn}

Let $n\geq 2$, then
\begin{enumerate}
\item[\rm{(i)}] \emph{[\textbf{Whitehead, \cite{whitehead_equivalent_1936}}]}
 $\WhP{F_n}{\Aut{F_n}}$ is solvable.
\item[\rm{(ii)}] \emph{[\textbf{Ciobanu-Houcine, \cite{ciobanu_monomorphism_2010}}]}
    $\WhP{F_n}{\Mon{F_n}}$ is solvable.
\item[\rm{(iii)}] \emph{[\textbf{Makanin, \cite{makanin_equations_1982}}]}
    $\WhP{F_n}{\End{F_n}}$ is solvable. \qed
\end{enumerate}
\end{theorem}

So, the auto, mono  and endo Whitehead problems (for words) are solvable for both $\ZZ^m$ and $F_n$. For $\fatfab = \fatf$ though, these problems turn out to be more than the mere juxtaposition of the corresponding problems for its factors. That is because the endomorphisms of $\fatfab$ are more than pairs of endomorphisms of $\ZZ^m$ and $F_n$ as well. It is not difficult to obtain a complete description of them imposing the preservation of the (commutativity) relations defining $\fatfab$.

\begin{proposition}\label{prop:endos}
The endomorphisms of $\fatfab = \fatf$ are of the form
\[
\Psi_{\phi, \mathbf{Q}, \mathbf{P}} \ \colon \ (\mathbf{a}, u) \ {\mapsto} \ (\mathbf{a Q + u P}, u \phi)
\]
where
$\mathbf{u}=u\ab \in \ZZ^n$,
$\mathbf{Q}$ and $\mathbf{P}$ are integer matrices,
and $\phi: F_n \to F_n$ is either
\begin{enumerate}
  \item [\rm{(i)}]
  an endomorphism of $F_n$, or
  \item [\rm{(ii)}]
  a map $u \mapsto w^{\alpha(\mathbf{u})}$
  where $w$ is a non-proper power word in $F_n \setminus \{1\}$ and\\
  $
  \alpha(\mathbf{u}) = \mathbf{a}\mathbf{l}\tr +\mathbf{u}\mathbf{h}\tr \in \ZZ
  $
  for certain  $\mathbf{l}\in \ZZ^m \setminus \{\mathbf{0}\}$ and $\mathbf{h}\in \ZZ^n$.
\end{enumerate}
We will refer to them as \emph{\ti} and \emph{\tii} endomorphisms of $\fatfab$ respectively.
\end{proposition}

Note that if $n=0$ then \ti\ and \tii\ endomorphisms do coincide. Otherwise, it turns out that \tii\ endomorphisms are a sort of degenerated case corresponding to a free contribution from the abelian part while all the injective and exhaustive endomorphisms of $\fatfab$ are of \ti. Indeed, viewing $\mathbf{Q}$ as the endomorphism
of $\ZZ^m$ given by right multiplying by $\mathbf{Q}$, we have the following quite natural characterization (note that the matrix $\mathbf{P}$
plays absolutely no role in this matter).
\begin{proposition}\label{prop:caract monos epis}
Let $\Psi$ be an endomorphism of $\fatfab = \fatf$, with $n\geq 2$. Then,
\begin{itemize}
\item[\rm{(i)}] $\Psi$ is a monomorphism if and only if it is of \ti\ with $\phi$ a monomorphism of $F_n$ and $\mathbf{Q}$ a monomorphism of $\ZZ^m$ (i.e.\ $\detm{\mathbf{Q}}\neq 0$).
\item[\rm{(ii)}] $\Psi$ is an epimorphism if and only if it is of \ti\ with $\phi$ an epimorphism of $F_n$ and $\mathbf{Q}$ an epimorphism of $\ZZ^m$ (i.e.\ $\detm{\mathbf{Q}} = \pm 1$).
\end{itemize}
\end{proposition}

The hopfianity of $\ZZ^m$ and $F_n$  together with this last proposition provide immediately the following results.

\begin{corollary}
$\fatf$ is hopfian and not cohopfian. \qed
\end{corollary}

\begin{corollary}
An endomorphism
of $\fatfab = \fatf$ ($n\geq 2$) is an automorphism if and only if it is of \ti\ with $\phi \in \Aut(F_n)$ and $\mathbf{Q} \in
    \GL{m}{\ZZ}$. \qed
\end{corollary}

Now we have the ingredients to prove the main result of this note.

\begin{theorem}\label{thm:Whitehead}
Let $\fatfab = \fatf$ with $m\geqslant 1$ and $n\geqslant 2$, then
\begin{itemize}
\item[\rm{(i)}] $\WhP{\fatfab}{\Aut{\fatfab}}$ is solvable,
\item[\rm{(ii)}] $\WhP{\fatfab}{\Mon{\fatfab}}$ is solvable,
\item[\rm{(iii)}] $\WhP{\fatfab}{\End{\fatfab}}$ is solvable.
\end{itemize}
\end{theorem}

\begin{proof}[Sketch of the proof]
We are given two elements $(\mathbf{a},u),\, (\mathbf{b},v) \in \fatfab$, and have
to decide whether there exists an automorphism (resp. monomorphism, endomorphism) of $\fatf$ sending one to the other, and in the affirmative case, find one of them.

Using the previous descriptions for each type of transformations in $\fatf$ and separating the free-abelian and free parts, our problems reduce to deciding whether there exist integer matrices $\mathbf{P},\mathbf{Q}$ and a transformation $\phi$ of $F_n$ ($\mathbf{Q}$ and $\phi$ of certain kind depending on the case, see proposition \ref{prop:caract monos epis}) such that the two following independent conditions hold.
\begin{empheq}[right=\,\empheqrbrace .]{align}
u\phi =v \label{eq:cond free}\\
\mathbf{aQ}+\mathbf{uP}=\mathbf{b} \label{eq:cond free-abelian}
\end{empheq}
Note that the subproblem associated to condition \eqref{eq:cond free} becomes respectively the already solved $\WhP{F_n}{\Aut{F_n}}$, $\WhP{F_n}{\Mon{F_n}}$ and $\WhP{F_n}{\End{F_n}}$ in the cases of autos, monos, and endos of \ti, and is straightforward to check for endos of \tii.
Thus, if there is not any $\phi$ solving these problems (for $F_n$) then our corresponding problem (for $\fatfab$) has no solution either, and we are done.

Otherwise, the decision method provides such a $\phi$ and our problem reduces to solving the subproblem associated to condition \eqref{eq:cond free-abelian}:
given arbitrary elements $\mathbf{a} \in \ZZ^m$ and $\mathbf{u}\in \ZZ^n$, decide whether there exist integer matrices $\mathbf{P}$ and $\mathbf{Q}$ (satisfying $\detm{\mathbf{Q}} \neq 0$ in the case of monos and $\detm{\mathbf{Q}} = \pm 1$ in the case of autos) such that $\mathbf{aQ}+\mathbf{uP}=\mathbf{b}$.

If $\mathbf{a}=\mathbf{0}$ or $\mathbf{u}=\mathbf{0}$, these are well known results in linear algebra, otherwise write
\comment{$\mathbf{a} = \alpha \mathbf{a'}$ and $\mathbf{u} = \mu \mathbf{u'}$, where}
$0\neq \alpha = \gcd(\mathbf{a})$ and $0\neq \mu =\gcd(\mathbf{u})$. Then the problems reduce to test whether the following linear system of equations
\begin{equation}\label{eq:sist dem witehead paraules simple}
\left.
\begin{array}{lclcc}
\alpha \, x_1 & + & \mu \, y_1 & = & b_1 \\
 & \vdots &  & \vdots & \\ \alpha \, x_m & + &\mu \, y_m & = & b_m \\
 \end{array}
 \right\}
\end{equation}
has integral solutions $x_1, \ldots,x_m,y_1,\ldots,y_m \in \ZZ$ (with no extra condition in the case of endos, satisfying $(x_1,\ldots,x_m)\neq \mathbf{0}$ in the case of monos, and satisfying ${\gcd (x_1,\ldots,x_m) = 1}$ in the case of autos).

So, for the case of endos the decision is a standard argument in linear algebra. In the case of monomorphisms the condition $(x_1,\ldots,x_m)\neq \mathbf{0}$ turns out to be superfluous and the same argument as for endos works, while the more involved case of autos became a not very difficult exercise in arithmetic and is decidable as well.

Finally, observe that in any of the afirmative cases,  we can easily reconstruct a transformation $\Psi$ (of the corresponding type) such that $(\mathbf{a},u) \Psi = (\mathbf{b},v)$.
\end{proof}

We note that, very recently, a new version of the classical peak-reduction theorem has been developed by M. Day~\cite{day_full-featured_2012} for an arbitrary partially commutative group (see also~\cite{day_peak_2009}). These techniques allow the author to solve the Whitehead problem for this kind of groups, in its variant relative to tuples of
conjugacy classes and automorphisms.
As far as we know, $\WhP{\groupi}{\Mon{\groupi}}$ and $\WhP{\groupi}{\End{\groupi}}$ remain unsolved for a general partially commutative group $G$. Our theorem~\ref{thm:Whitehead} is a small contribution into this direction, solving these problems for free-abelian times free groups in a direct and self-contained form.

\bibliography{mybib}{}
\bibliographystyle{acm}

\end{document}